\documentclass[11pt]{amsart} 
\usepackage{amsmath,amsfonts,amssymb,amsthm,amscd,fancyhdr,enumerate,mathrsfs,epsfig,subfigure,verbatim,ifthen}
\usepackage{graphicx}

\theoremstyle{plain} \numberwithin{equation}{subsection}
\newtheorem{theorem}{Theorem}[section]
\newtheorem{corollary}[theorem]{Corollary}

\newtheorem{lemma}[theorem]{Lemma}
\newtheorem{proposition}[theorem]{Proposition}
\newtheorem{proposition/definition}[theorem]{Proposition/Definition}

\theoremstyle{definition}
\newtheorem{remark}[theorem]{Remark}

\newtheorem{definition}[theorem]{Definition}

\newtheorem{notation}[theorem]{Notation}

\def\eb#1#2{{\mathcal{L}_{#2} (#1)}} 
\def\eB#1#2{{\mathcal{L}_{#2}\big(#1\big)}} 
\def\EB#1#2{{\mathcal{L}_{#2}\Big(#1\Big)}} 

\def\fcohsymdegs#1#2#3#4{{H^{#3} \Big(G/B, \: \mathcal{L}_B\big(#1\big) \otimes \mathcal{L}_B\big(\symmpg{#2}{#4} \big) \Big)}} 

\def\pc#1#2#3{{  H^{#1, \nil{#3}} (#2) }} 

\def\fc#1#2{{H^{#1, \mf{n}} (#2)}} 

\def\fcot{{T^*(G/B)}} 
\def\pcot#1{{T^*(G/#1)}} 

\def\cohom#1#2#3{{H^{#1}  ( #2, \, #3 )}} 
\def\Cohom#1#2#3{{H^{#1} \big( #2, \: #3 \big)}} 
\def\COhom#1#2#3{{H^{#1} \Big( #2, \: #3 \Big)}} 

\def\fcotcoh#1{{H^i \big(\fcot, \: p_B^* \, \mathcal{L}_B(#1) \big)}} 
\def\pcotcoh#1#2{{H^i \big(\pcot{#2}, \: p_{#2}^* \, \mathcal{L}_{#2}(#1) \big)}} 

\def\nil#1{{\mf{n}_{\mf{#1}}}} 
\def\symm{{S(\mf{n}^*)}} 
\def\symmp#1{{S( \mf{n_{#1}^*})}} 
\def\symmpg#1#2{{S^{#2}(\nil#1^*)}} 
\def\nc{{\mathcal{N}}} 

\def\mf#1{{\mathfrak{#1}}} 
\def\mc#1{{   \mathcal{#1}   }}
\def\mb#1{{   \mathbb{#1}   }}

\def\chara#1{{\mathbb{C}_{#1}}} 
\def\irrep#1#2{{V^{#2}(#1)}} 
\def\irrepd#1#2{{V^{#2}(#1)^*}} 
\def\irrepg#1{{V(#1)}}
\def\irrepgd#1{{V(#1)^*}}
\def\irrepgw#1#2{{    V_{#1}(#2)     }}

\def\Lhi#1#2#3{{   \mc{W}_{#1}^{#3}(#2)   }}
\def\irrepfilt#1#2#3#4{{     \mathcal{H}^{#4}  \big(   \Lhi{#1}{#2}{#3}   \big)   }}

\def\irrepbrygen#1#2#3#4#5{{     \mathcal{F}^{#4}_{#5}  \big(   \Lhi{#1}{#2}{#3}  \big)   }}
\def\bry#1#2#3{{   \mc{F}^{#3} \big(   \irrepgw{#1}{#2}   \big)   }}

\def\res#1{{ \substack{\phantom{1} \\ \big|}_{#1}}} 
\def\struct#1{{\mathcal{O}_{#1}}} 

\def\C{{\mathbb{C}}} 
\def\pcotpp#1{{\EB{\symmp{p} \otimes \irrepd{#1}{P}}{P}}} 
\def\glpp#1{{\eB{\C[H + \nil{p}] \otimes \irrepd{#1}{P}}{P}}} 

\def\pcotppgr#1#2{{\EB{S^{#2}(\nil{p}^*) \otimes \irrepd{#1}{P}}{P}}} 

\def\shF#1#2{{\mc{L}_{#2}(#1)   }}  
\def\shL#1#2{{\mc{A}_{#2}(#1)    }} 
\def\shLpush#1#2{{  q_*\, \shL{#1}{#2}  }} 
\def\shLfilt#1#2#3{{ \big(\shLpush{#1}{#2}\big)^{\leq#3}    }} 
\def\shcotpush#1#2{{  p_* \, p^* \shF{#1}{#2} }} 
\def\shcotgr#1#2#3{{ \big( \shcotpush{#1}{#2} \big)^{#3}  }} 
\def\cohomL#1#2{{   \cohom{#1}{G/P}{   \shLpush{#2}{P}   }    }} 
\def\cohomLfilt#1#2#3{{      \COhom{#1}{G/P}{   \shLfilt{#2}{P}{#3}  }    }}  
\def\aff{{  H + \nil{p}  }}
\def\regaff{{  \C[\aff]  }}
\def\regafffilt#1{{  \regaff^{\leq #1}  }}

\def\congsp{{  \: \: \cong \: \: }} 

\def\isom{{\ \tilde{\longrightarrow} \ }} 
\def\shortexact#1#2#3{{0 \: \rightarrow \: #1 \: \rightarrow \: #2 \: \rightarrow \: #3 \: \rightarrow \: 0}} 
\def\cht#1{{\textrm{cht} \left(#1\right)}} 
\def\pair#1#2{{   #1(#2^\vee)   }} 
\def\pairp#1#2{{         \left(#1\right)\left(#2^\vee\right)         }} 

\def\normsq#1{{    \left| #1 \right|^2    }}
\def\inner#1#2{{    \left(#1 \,, #2 \right)    }}
\def\rootlat{{    \Lambda_R    }}

\def\orb#1{{   \mf{O}_{#1}   }} 
\def\part{{   \underline{\textbf{p}}   }} 
\def\pkos{{   p^P_q   }} 
\def\pkoscoeff#1{{   p^{\, n, P}_q   }} 
\def\qana#1#2{{   m_{#1}^{P, #2} (q)   }} 
\def\qanacoeff#1#2#3{{ m_{#1}^{#3, P, #2}   }}
\def\chichar#1{{   \chi \big(\textrm{ch} \, #1 \big)   }}

\def\algclosed#1{{   \bar{  \mb{F}  }_{#1}   }}
\def\r#1{{   r_{\alpha_{#1}}   }}

\begin{document}

\bibliographystyle{amsplain}

\title{Cohomology of Flag Varieties and the Brylinski-Kostant Filtration}
\author{Chuck Hague}
\begin{abstract} Let $G$ be a reductive algebraic group over $\mb C$ and let $N$ be a $G$-module. For any subspace $M$ of $N$, the Brylinski-Kostant filtration on $M$ is defined through the action of a principal nilpotent element in Lie$\, G$. This filtration is related to a $q$-analog of weight multiplicity due to Lusztig. We generalize this filtration to other nilpotent elements and show that this generalized filtration is related to "parabolic" versions of Lusztig's $q$-analog of weight multiplicity. Along the way we also generalize results of Broer on cohomology vanishing of bundles on cotangent bundles of partial flag varieties. We conclude by computing some explicit examples.
\end{abstract}
\maketitle

\section{Introduction}


Let $G$ be a semisimple algebraic group over $\C$. Given an irreducible representation $V$ of $G$, Brylinski \cite{Bry89} constructed a filtration on weight spaces of $V$. This filtration, often referred to as the Brylinski-Kostant filtration, studies the action on $V$ of certain special nilpotent elements of the Lie algebra of $G$ (the \textbf{principal} nilpotents). This filtration was motivated by fundamental work of Kostant \cite{K59}, \cite{K63} on adjoint orbits and actions of $sl_2$-triples. Using higher cohomology vanishing of pullback bundles on the cotangent bundle of the flag variety $G/B$, Brylinski showed that certain polynomials -- Lusztig's $q$-analogs of weight multiplicity -- compute the weight spaces occuring in the various degrees of this filtration. These polynomials, as proven in \cite{Ka82}, are \textbf{Kazhdan-Lusztig polynomials}, which are deep objects in combinatorial representation theory (see \cite{KL79b}, \cite{KL79a}, and \cite{L83} for a starting point in the large amount of literature on this subject).

We now consider a larger class of nilpotent elements, the \textbf{even} nilpotents. Given an even nilpotent $X$ we define a generalized Brylinski-Kostant filtration corresponding to $X$. If $X$ is principal we obtain the original Brylinski-Kostant filtration. As in the original filtration, certain polynomials --  in this case, parabolic versions of Lusztig's $q$-analogs -- arise in computing dimensions of the generalized filtration. In particular, these polynomials arise in computing multiplicities of irreducible $G$-modules in the global sections of certain equivariant bundles on cotangent bundles of partial flag varieties; we then relate these multiplicities to the generalized Brylinski-Kostant filtration.

Our results rely on cohomology vanishing results for equivariant bundles on cotangent bundles of flag varieties. In section \ref{sec:cohom} we obtain extensions of some results of Broer \cite{Bro94}. These vanishing results provide a basic technical tool in analyzing the generalized Brylinski-Kostant filtration but are interesting in their own right; another application of these cohomology vanishing results is in the geometry of $G$-orbits of nilpotent elements in Lie($G$). These orbits and their closures are subvarieties of Lie($G$) with rich geometric structure, and these cohomology vanishing results are a vital tool in studying their structures, cf \cite{Bro94}, \cite{H79}, \cite{S03}, and \cite{S05}. 

In section \ref{sec:examples} we conclude by explicitly giving examples of generalized Brylinski-Kostant filtrations on various irreducible representations.

\bigskip

\textbf{Acknowledgements.} I would like to thank Shrawan Kumar and George McNinch for many useful conversations, and the referee for a number of helpful comments, including a new proof of Theorem \ref{th:vanishing} that strengthens the result that had been there previously.

\section{Notation} \label{sec:notation} Let $G$ denote a complex semisimple simply-connected algebraic group over $\C$. Fix a Borel $B$ and maximal torus $T \subseteq B$ of $G$ (we assume $G$ simply-connected only for notational convenience; all results will generalize easily to $G$ of arbitrary isogeny type). Let $U$ denote the unipotent radical of $B$. Set $\mf{g} := \textrm{Lie } G$, $\mf{h} := \textrm{Lie } T$, and $\mf{n} := \textrm{Lie } U$. Let $W$ be the Weyl group of $G$. Let $\Lambda \subseteq \mf{h}^*$ denote the weight lattice of $G$. Let $\Lambda_R \subseteq \Lambda$ denote the root lattice. Let $\Delta$ (resp. $\Delta^-$, $\Delta^+$), denote the roots (resp. positive and negative roots) with respect to $T$ and $B$ and let $\pi \subseteq \Delta^+$ denote the simple roots. Let $\rho$ be the half sum of all elements of $\Delta^+$. There is a shifted action of $W$ on $\mf{h}^*$ defined by $w*\lambda := w(\lambda + \rho) - \rho$. This action keeps $\Lambda$ and $\Lambda_R$ stable. 


For each $\beta \in \Delta$ let $\beta^\vee$ denote the coroot corresponding to $\beta$. Set $$D := \{     \lambda \in \Lambda : \pair{\lambda}{\alpha} \geq 0 \textrm{ for all } \alpha \in \pi  \},$$ the collection of \textbf{dominant weights} in $\Lambda$. Also set $$D' := \{    H \in \mf{h} : \alpha(H) \in \mb{R} \textrm{ and } \alpha(H) \geq 0 \textrm{ for all } \alpha \in \pi  \},$$ the \textbf{dominant chamber} in $H$. For any $\mu \in D$ let $\irrepg{\mu}$ denote the irreducible representation of $G$ with highest weight $\mu$.

For any $\alpha \in \pi$ let $\chi_\alpha$ denote the fundamental weight corresponding to $\alpha$. That is, $\pair{\chi_\alpha}{\alpha} = 1$ and $\pair{\chi_\alpha}{\beta} = 0$ for all $\beta \in \pi \setminus \{  \alpha  \}$. We say that a weight $\mu \in \Lambda$ is \textbf{regular} if $\pair{\mu}{\beta} \neq 0$ for all $\beta \in \Delta^+$. 




For any $T$-module $M$ and any weight $\mu$ of $M$, let $M_\mu$ denote the $\mu$-weight subspace of $M$. Set $$\irrepgw{\lambda}{\mu} := \irrepg{\mu}_\lambda \, .$$ 

Let $P$ be a standard parabolic in $G$ and let $V$ be a variety with a $P$-action. Let $G \times^P V$ denote the $G$-equivariant fiber bundle on $G/P$ with fiber $V$.  Let $$f : G \times V \, \twoheadrightarrow \, G \times^P V$$ be the quotient map. For $g \in G$ and $x \in V$ we set $$g * x := f(g, x) \, .$$ Then $gp*x = g*px$, for all $g \in G$, $p \in P$, and $x \in V$. For a $P$-module $M$, $G \times^P M$ is a $G$-equivariant vector bundle on $G/P$. Let $\eb{M}{P}$ denote the sheaf of sections of this bundle.

Let $L$ denote the Levi factor of $P$ containing $T$. Given $\lambda \in \Lambda$, we say that $\lambda$ is \textbf{$P$-dominant} (resp. $P$-\textbf{regular dominant}) if $\pair{\lambda}{\beta} \geq 0$ (resp. $\pair{\lambda}{\beta} > 0$) for all roots $\beta$ of $L$. Note that for $P = B$ these conditions are vacuous. For $P$-dominant $\lambda$ let $\irrep{ \lambda }{ P }$ denote the irreducible $P$-module of highest weight $\lambda$, and set $$\eb{\lambda}{P} := \eB{ \irrepd{\lambda}{P} }{ P } \, \cong \, \eB{ \irrep{\lambda}{P} }{ P }^\vee \, .$$ Note that $\eb{ \lambda }{B}$ is a line bundle on $G/B$ for all $\lambda \in \Lambda$.

\section{The BK-filtration and generalizations} \label{sec:bk}

\subsection{Definitions and background}

\subsubsection{Definitions}

\begin{definition}
A triple $\{   X, Y, H   \} \subseteq \mf{g}$ is called an \textbf{$sl_2$-triple} if there is a Lie algebra isomorphism $sl_2 \, \isom \, \textrm{span}\{   X, Y, H   \}$ such that the standard basis $\{   X', Y', H'   \}$ of $sl_2$ maps to $\{   X, Y, H   \}$. See Section \ref{subsubsec:nilpotent orbits} below for more information on $sl_2$-triples.

Let $\nc$ denote the nullcone (the variety of all nilpotent elements) of $\mf{g}$. We say that a nilpotent element $X \in \nc$ is in \textbf{good position} if (1) $X \in \mf{n}$ and (2) there is an $sl_2$-triple $\{  X, Y, H  \} \subseteq \mf{g}$ with $H \in D'$. There is a unique dense open $G$-orbit in $\nc$; elements of this orbit are called \textbf{principal} nilpotents.
\end{definition}

\begin{remark} Let $\{ X_\beta \}_{\beta \in \Delta^+}$ be a Chevalley basis of $\mf{n}$. Then the following are equivalent (cf \cite{Ja04}, Proposition 4.14): 

\begin{enumerate} \item $X \in \nc$ is a principal nilpotent in good position.
\item The $B$-orbit of $X$ is dense in $\mf{n}$. 
\item We have $$X = \sum_{\beta \in \Delta^+} c_\beta \, X_\beta \, ,$$ where the $c_\beta$ are constants such that $c_\alpha \neq 0$ for all $\alpha \in \pi$.
\end{enumerate}
\end{remark}

\begin{definition} The \textbf{BGG category} $\mc{O}$ consists of $\mf{g}$ representations V that have finite-dimensional weight spaces and are \textbf{$\mf{n}$-locally finite}; i.e., for any $v \in V$, $v$ lies in a finite-dimensional $\mf{n}$-submodule of $V$.
\end{definition}

Let $V$ be an object in the BGG category $\mc{O}$ and let $U \subseteq V$ be any vector subspace. Let $e$ be a principal nilpotent in good position. In \cite{Bry89}, Ranee Brylinski defined a filtration on $U$ inside of $V$, called the \textbf{Brylinski-Kostant filtration} (or \textbf{BK-filtration}), as follows: set $$\mc{F}^n_V(U) := \{v \in U : e^{n+1} v = 0 \} \, .$$

In particular, Brylinski considered this filtration in the case where $V$ is an irreducible representation $\irrepg{\mu}$ of $G$ and $U$ is a weight subspace $\irrepgw{\lambda}{\mu}$ of $V$, for $\mu \in D$ and $\lambda \in \Lambda$. In this case we shall suppress the subscript and write \begin{equation*} \label{eq:BK} \bry{\lambda}{\mu}{n} := \mc{F}^n_{\irrepg{\mu}} \big(   \irrepgw{\lambda}{\mu}   \big) \, . \end{equation*}

Define $$r_\mu^\lambda (q) := \sum_{n \geq 0} \mathrm{dim} \left( \frac{    \bry{\lambda}{\mu}{n}    }{    \bry{\lambda}{\mu}{n-1}    } \right) q^n \in \mb{Z}[q]$$ (where $   \bry{\lambda}{\mu}{-1}    = \{0\}$). This \textbf{jump polynomial} counts the dimensions of the degrees of the filtration, and clearly $\textrm{dim} \, (\irrepgw{\lambda}{\mu}) = r_\mu^\lambda (1)$.

\subsubsection{Kazhdan-Lusztig polynomials and Brylinski's theorem}

For $\gamma \in \Lambda_R$ let $p_q(\gamma) \in \mb{Z}[q]$ be the coefficient of $e^\gamma$ in $\displaystyle \prod_{\beta \in \Delta^+} (1 - qe^{\beta})^{-1}$. This is a q-analog of Kostant's partition function $p(\gamma) = p_q(\gamma) \res{q=1} \, $, which counts the number of ways of writing $\gamma$ as a sum of positive roots. The degree-$n$ coefficient of $p_q(\gamma)$ counts the number of ways of writing $\gamma$ as a sum of precisely $n$ (not necessarily distinct) positive roots.

For $\mu$, $\lambda \in D$ such that $\lambda - \mu \in \Lambda_R$ set $$\displaystyle m^{\lambda}_{\mu}(q) := \sum_{w \in W} (-1)^{l(w)} p_q(w*\mu - \lambda) \, ,$$ Lusztig's \textbf{q-analog of weight multiplicity} \cite{L83}. It is called an analog of weight multiplicity because, by the Weyl character formula, $m_\mu^\lambda(1) = \textrm{dim} \, (\irrepgw{\lambda}{\mu})$. The polynomials $m_{\mu}^{\lambda}(q)$ for $\lambda$, $\mu \in D$ are equal to certain \textbf{Kazhdan-Lusztig polynomials} for affine Weyl groups, which are important objects in combinatorial representation theory (see \cite{Ka82}, \cite{KL79b}, \cite{KL79a}, and \cite{L83}).


Using this setup, we have the following theorem of Brylinski which relies on a cohomology vanishing condition. Let $p_B: \fcot \rightarrow G/B$ be the cotangent bundle of $G/B$.

\begin{theorem} \label{th:bry} (Brylinski, \cite{Bry89}) Let $e$ be a principal nilpotent in good position. Let $\mu$, $\lambda \in D$. If $$\fcotcoh{\lambda} = 0$$ for all $i > 0$ then $m_{\mu}^{\lambda}(q) = r_\mu^\lambda (q)$. 
\end{theorem}

We also have the following result of Broer which is a special case of Theorem \ref{th:broer} below.

\begin{theorem} (Broer, \cite{Bro93}) $\fcotcoh{\lambda} = 0$ for all $\lambda \in D$ and $i > 0$.

\end{theorem}

Thus we have

\begin{corollary} For $\mu$, $\lambda \in D$, $m^{\lambda}_{\mu}(q) = r^\lambda_\mu (q)$. 

\end{corollary}

\bigskip

\subsection{Generalization of the BK-filtration}

\subsubsection{Nilpotent Orbits} \label{subsubsec:nilpotent orbits}

In this section we collect some basic facts about $sl_2$-triples and nilpotent orbits in characteristic 0. Our main reference for this section is \cite{CM}. It should be noted that many aspects of the theory change in positive characteristic; for a good overview of the theory in arbitrary characteristic, see \cite{Ja04}.

Since an $sl_2$-triple gives $\mf{g}$ the structure of an $sl_2$-representation, we see that if $\{   X, Y, H   \}$ is an $sl_2$-triple then $X$ and $Y$ are nilpotent and $H$ is semisimple. $X$ is called the \textbf{nilpositive} element of the triple and $H$ is the \textbf{semisimple} element of the triple. We have the following important theorem.

\begin{theorem} (Jacobson-Morozov) Let $X \in \nc \setminus 0$. Then there is an $sl_2$-triple through $X$, ie an $sl_2$-triple $\{   X, Y, H   \}$.
\end{theorem}


\bigskip

Let $\{   X, Y, H   \}$ be an $sl_2$-triple. Then, viewing $\mf{g}$ as an $sl_2$-representation, we can write the eigenspace decomposition of $\mf{g}$ under the action of $H$ as $$\mf{g} = \bigoplus_{n \in \mathbb{Z}} \mf{g}_n \, .$$ This decomposition depends only on the nilpositive element $X$ and not on the rest of the elements of the triple. Thus, given any nilpotent element $X$, we obtain such a decomposition of $\mf{g}$ into eigenspaces. We say that $X$ is \textbf{even} if $\mf{g}_n = 0$ for all odd $n$.

One may check that $\displaystyle \bigoplus_{n \geq 0} \mf{g_n}$ is a parabolic subalgebra of $\mf{g}$; we call this the parabolic subalgebra \textbf{associated to} $X$ (or $H$). It has Levi factor $\mf{g}_0 = \mf{g}^H$ and nilradical $\displaystyle \bigoplus_{n > 0} \mf{g}_n$ (where $\mf{g}^H$ denotes the centralizer of $H$ in $\mf{g}$). This associated parabolic is a standard parabolic iff $X$ is in good position, and $X$ is regular iff the associated parabolic is a Borel subalgebra of $\mf{g}$. Let $P \subseteq G$ be the parabolic subgroup of $G$ such that Lie$(P) = \mf{p}$; we will call $P$ the parabolic subgroup \textbf{associated to} $X$. 

Let $\mf{p}$ be a parabolic subalgebra of $\mf{g}$. For $X \in \nc$ let $\orb{X}$ denote the $G$-orbit of $X$. We say that $X$ is a \textbf{Richardson} element for $\mf{p}$ if $\mf{O}_X \cap \nil{p}$ is dense in $\nil{p}$. The Richardson elements for $\mf{p}$ form a unique orbit in $\nc$. Furthermore, every parabolic $\mf{p}$ has a Richardson element.

Each semisimple orbit in $\mf{g}$ has a unique element in $D'$. If $H$ is a semisimple element in $D'$ such that $H$ is in some $sl_2$-triple $\{  X, Y, H  \}$ then we say that $H$ is \textbf{distinguished semisimple}. Further, if $H \in D'$ is distinguished semisimple then $\alpha(H) \in \{  0, 1, 2  \}$ for all $\alpha \in \pi$. If $X \in \nc$ is in good position then there is an $sl_2$-triple $\{  X, Y, H  \}$ such that $H$ is distinguished.

\begin{lemma} \label{lem:kos} (Kostant) Let $X \in \nc$ and let $\mf{p}$ be the associated parabolic. Let $P \subseteq G$ be the parabolic subgroup with Lie algebra $\mf{p}$. Set $$\mf{p}_2 := \displaystyle \bigoplus_{n \geq 2} \mf{g}_n \subseteq \nil{p} \, .$$ Then $$\mf{O}_X \cap \mf{p}_2 = \orb{X} \cap \nil{p} = P.X$$ and $\mf{O}_X \cap \mf{p}_2$ is open dense in $\mf{p}_2$.
\end{lemma}

\begin{remark} \label{rem:even and richardson} Lemma \ref{lem:kos} implies that $X$ is even nilpotent iff $X$ is a Richardson element for its associated parabolic, since $X$ is even iff $\mf{p}_2 = \nil{p}$.
\end{remark}

\subsubsection{Definitions and main result} \label{sec:gendefs}
We now generalize the BK-filtration to the case of an arbitrary nilpotent element. Choose $X \in \nc$ and $V$ in the category $\mc{O}$. Let $U \subseteq V$ be any vector subspace. Define a filtration $\mathcal{F}_{X, V}$ on $U$ by $$\mathcal{F}_{X, V}^n (U) := \{     v \in U :   X^{n+1} v = 0   \} \, .$$ If $e$ is a principal nilpotent in good position then clearly $\mathcal{F}_{e, V}^n (U) = \mathcal{F}_V^n (U)$, the original BK-filtration.

Let $X \in \mf{n}$ be an even nilpotent element in good position and let $P$ be its associated standard parabolic. Let $L$ be the Levi factor of $P$ containing $T$. Let $\Delta_P$, $\Delta^+_P$, and $\pi_P$ denote the roots, the positive roots, and the simple roots, respectively, of $L$. For any $\mu \in D$ and weight $\lambda$ of $\irrepg{\mu}$ let $\Lhi{\lambda}{\mu}{P} \subseteq \irrepgw{\lambda}{\mu}$ denote the subspace consisting of $L$-highest weight vectors.

Set $$\irrepbrygen{\lambda}{\mu}{P}{n}{X} := \, \mathcal{F}_{X, \irrepg{\mu}}^n \big( \Lhi{\lambda}{\mu}{P} \big) \, .$$ This generalizes the filtration $\mc{F} \big(\irrepgw{\lambda}{\mu}\big)$ from (\ref{eq:BK}) above. Note that $\Lhi{\lambda}{\mu}{B} = \irrepgw{\lambda}{\mu}$, so that $$\irrepbrygen{\lambda}{\mu}{B}{n}{e} = \, \bry{\lambda}{\mu}{n} \, .$$ We obtain a jump polynomial as before: set $$r_\mu^{X, \lambda} (q) := \sum_{n \geq 0} \mathrm{dim} \left( \frac{    \irrepbrygen{\lambda}{\mu}{P}{n}{X}    }{    \irrepbrygen{\lambda}{\mu}{P}{n-1}{X}    } \right) q^n \: .$$

Let $\pkos (\gamma) \in \mb{Z}[q]$ be the coefficient of $e^\gamma$ in $\displaystyle \prod_{\beta \in \Delta^+ \setminus \Delta^+_P} (1 - qe^{\beta})^{-1}$ and set $$\qana{\mu}{\lambda} := \displaystyle \sum_{w \in W} (-1)^{l(w)} \pkos (w*\mu - \lambda) \, .$$

\begin{remark} Note that $p_q^B(\gamma) = p_q(\gamma)$ and $m_\mu^{B, \, \lambda}(q) = m_\mu^\lambda(q)$. Thus the polynomials $\qana{\mu}{\lambda}$ are $P$-generalized versions of Kazhdan-Lusztig polynomials.
\end{remark}

Let $p_P: \pcot{P} \rightarrow G/P$ be the cotangent bundle of $G/P$. The following theorem is the main theorem in this section; it generalizes Theorem \ref{th:bry} above.

\begin{theorem} \label{th:brygen} Let $X \in \mf{n}$ be a standard even nilpotent. Let $\mu$, $\lambda \in D$. If $$\pcotcoh{\lambda}{P} = 0$$ for all $i > 0$, then $r_\mu^{X, \lambda} (q) = \qana{\mu}{\lambda}$, where $P$ is the parabolic associated to $X$.
\end{theorem}

In section \ref{sec:cohom} below, we will prove this cohomology vanishing for various $\lambda$.

\begin{remark} Note that this implies, in the presence of the cohomology vanishing condition, that the coefficients of $\qana{\mu}{\lambda}$ are positive when $\mu, \, \lambda \in D$. This is by no means clear from the definition and, in fact, can fail if $\lambda \notin D$.
\end{remark}

\begin{remark} One can note that in the proof of Theorem \ref{th:brygen} below, we do not explicitly use the fact that $X$ is a standard even nilpotent; we only use the following a priori weaker conditions: (i) There is $H \in D'$ with $[H, X] = X$; and (ii) $X$ is Richardson for the parabolic associated to $H$. However, these conditions are equivalent to $X$ being even nilpotent. Indeed, by the proof of the Jacobson-Morozov theorem, if $H'$ is a semisimple element of $\mf{g}$ such that $[ H', X ] = 2X$, then $H'$ is the semisimple element in some $sl_2$-triple containing $X$. Furthermore, if $X$ is Richardson for the corresponding associated parabolic, then $X$ is even by Remark \ref{rem:even and richardson}.
\end{remark}

\bigskip

\subsection{Proof of Theorem \ref{th:brygen}}

\subsubsection{Outline of proof}
\begin{remark} Our proof is an adaptation of Brylinski's proof of her Theorem 3.4 in \cite{Bry89} (which is stated as Theorem \ref{th:bry} above).
\end{remark}

Let us first give a sketch of the proof. Choose an even nilpotent element $X$ in good position and let $P$ be the parabolic subgroup corresponding to $X$. Let $\lambda$ be a $P$-dominant weight and let $$p: \pcot{P} \, \rightarrow \, G/P$$ be the cotangent bundle of $G/P$. We first construct two $G$-equivariant locally-free sheaves $\shLpush{\lambda}{P}$ and $\shcotpush{\lambda}{P}$ (defined below) of infinite rank on $G/P$ such that (a) $\shLpush{\lambda}{P}$ is filtered by $G$-equivariant locally-free sheaves $\shLfilt{\lambda}{P}{n}$ of finite rank and (b) $\shcotpush{\lambda}{P}$ is graded by $G$-equivariant locally-free sheaves $\shcotgr{\lambda}{P}{n}$ of finite rank.

We then show in Theorem \ref{th:brygenfilt} below that for any $\mu \in D$ the multiplicity of $\irrepgd{\mu}$ in $$\frac{\COhom{0}{G/P}{\shLfilt{\lambda}{P}{n}}}{\COhom{0}{G/P}{\shLfilt{\lambda}{P}{n-1}}}$$ is precisely the degree-n coefficient of the jump polynomial $r_\mu^{X, \lambda} (q)$. In Proposition \ref{lem:charcomp} we show that the cohomology vanishing condition implies that  for all $n \geq 0$ the multiplicity of $\irrepgd{\mu}$ in $\COhom{0}{G/P}{  \shcotgr{\lambda}{P}{n}  }$ is the degree-$n$ coefficient of $\qana{\mu}{\lambda}$. To complete the proof of Theorem \ref{th:brygen}, we show that the cohomology vanishing condition gives an isomorphism $$\frac{\COhom{0}{G/P}{\shLfilt{\lambda}{P}{n}}}{\COhom{0}{G/P}{\shLfilt{\lambda}{P}{n-1}}} \congsp \COhom{0}{G/P}{  \shcotgr{\lambda}{P}{n}  }$$ of $G$-modules.

\subsubsection{The proof}

For the remainder of this section fix $X, P$ and $\lambda$ as above. Let $L$ be the Levi factor of $P$ containing $T$. Recall that $D' \subseteq \mf{h}$ is the dominant chamber. Let $H' \in D'$ be the unique distinguished semisimple element in $D'$ occuring in an $sl_2$-triple with nilpositive element $X$ and set $H := H' / 2$. Then $$G^{H'} = G^H = L$$ and $$\mf{g}^{H'} = \mf{g}^H = \mf{l} \, ,$$ where $G^H$ is the stabilizer of $H$ in $G$ and $\mf{g}^H$ is the stabilizer of $H$ in $\mf{g}$. Note that we also have $[H, X] = X$. 

Let $G/L \twoheadrightarrow G/P$ be the projection map. Via this map we have a $G$-equivariant isomorphism $G/L \, \cong \, G \times^P (P/L)$ and hence we will consider $G/L$ as a $G$-equivariant fiber bundle on $G/P$.

\begin{lemma} Let $P$ act on $H + \nil{p}$ through the adjoint action. Then there is an isomorphism $$G/L \,\cong \, G \times^P (H + \nil{p})$$ of $G$-equivariant fiber bundles on $G/P$.
\end{lemma}

\begin{proof}
It suffices to show that $P/L \, \cong \, H + \nil{p}$ as $P$-varieties. As a variety, $P \cong U_P \times L$. Since $U_P^H = \{e\}$, we have $P^H = L$ and hence $P / L \, \cong \, P.H $.  Since $L.H = H$ we have $P.H = U_P.H$ and hence a variety isomorphism $U_P \: \tilde{\longrightarrow} \: U_P.H$.

Note that $\nil{p}.H \subseteq \nil{p}$, so that $U_P.H \subseteq H + \nil{p}$ (cf \cite{CG}, Lemma 1.4.12(i)). As $U_P$ and $H + \nil{p}$ are both isomorphic to $\mathbb{A}^n$ for $n = \textrm{dim}(\nil{p})$, the  variety injection $U_P.H \hookrightarrow H + \nil{p}$ must be an isomorphism.  Thus we have $$P/L \, \cong \, P.H = \, U_P.H \, \cong \, H + \nil{p}$$ as desired.
\end{proof}

\begin{remark} From here on, we will use the $G$-equivariant isomorphism $G/L \, \cong \, G \times^P (H + \nil{p})$ to write elements of $G/L$ in the form $g * (H + Z)$, for $g \in G$ and $Z \in \nil{p}$.
\end{remark}


\begin{notation} Let $p: \pcot{P} \rightarrow G/P$ and $q: G/L \rightarrow G/P$ be the bundle maps and set $$\shL{\lambda}{P} := q^* \eb{\lambda}{P}   \, .$$ 
\end{notation}



\begin{remark} \label{rem:function interpretation} By the projection formula, since $$\pcot{P} \congsp G \times^P \nil{p}$$ and $$G/L \congsp G \times^P (H + \nil{p}) \, ,$$ we have $$\shcotpush{\lambda}{P} \congsp \pcotpp{\lambda}$$ and $$\shLpush{\lambda}{P} \congsp \glpp{\lambda} \, .$$

\end{remark}

\begin{definition} \label{def:gr} Define a gradation on $\shcotpush{\lambda}{P}$ by $$\shcotgr{\lambda}{P}{n} := \pcotppgr{\lambda}{n}$$ for all $n \geq 0$.


For any $f \in \regaff$ there is a well-defined notion of the top degree of $f$, which is a homogeneous polynomial on $\nil{p}$. Define the \textbf{degree} of $f$ to be the degree of this polynomial. This gives rise to a natural $P$-invariant filtration $\regafffilt{n}$ on $\regaff$. We now obtain a natural $P$-invariant filtration $$\regaff ^{\leq n}\otimes \irrepd{\lambda}{P}$$ on $\regaff \otimes \irrepd{\lambda}{P}$ which gives rise to a $G$-filtration $$\shLfilt{\lambda}{P}{n} := \eB{  (\regaff^{\leq n} \otimes \irrepd{\lambda}{P})  }{P}$$ on $\shLpush{\lambda}{P}$.

\end{definition}

\begin{proposition} \label{pr:grL} There is a natural $G$-equivariant isomorphism $$\textrm{gr } \shLpush{\lambda}{P} \isom \shcotpush{\lambda}{P}$$ of graded sheaves on $G/P$.
\end{proposition}

\begin{proof} This follows immediately from an argument similar to that of Theorem 5.5 of \cite{Bry89}.
\end{proof}

Recall that if $\lambda$ is a weight of $\irrepg{\mu}$ then $\Lhi{\lambda}{\mu}{P}$ denotes the subspace of $\irrepgw{\lambda}{\mu}$ consisting of $L$-highest weight vectors. If $\lambda$ is not a weight of $\irrepg{\mu}$ we set $\Lhi{\lambda}{\mu}{P} = \{  0  \}$. Note that $\Lhi{\lambda}{\mu}{B} = \irrepgw{\lambda}{\mu} \, .$

\begin{remark} \label{rem:frobrecip}
Let $B_L$ denote the Borel subgroup of $L$. By Proposition \ref{pr:grL} and Frobenius reciprocity, we have isomorphisms of $B_L$-modules
\begin{eqnarray} 
Hom_G \Big(   \irrepg{\mu}^*   , \,   \cohomL{0}{\lambda}   \Big) & \cong & Hom_G \Big(   \irrepg{\mu}^*   , \,   \cohom{0}{G/L}{   \shL{\lambda}{P}   }   \Big) \nonumber \\
& \cong & Hom_L \big(  \irrepg{\mu}^*  , \,  \irrepd{\lambda}{P}   \big) \nonumber \\
& \cong & Hom_{B_L} \big(    \irrepg{\mu}^*    , \,   \chara{-\lambda}   \big) \nonumber \\
& \cong & Hom_{B_L} \big(     \chara{\lambda}  , \,  \irrepg{\mu}  \big) \label{eqn:repisoms} \\
& \cong & \Lhi{\lambda}{\mu}{P} \nonumber \\
& \subseteq & \irrepg{\mu} \nonumber  \, .
\end{eqnarray}
Denote by $\varphi$ the isomorphism $$Hom_G \Big(   \irrepg{\mu}^*   , \,   \cohomL{0}{\lambda}    \Big) \, \isom \, \Lhi{\lambda}{\mu}{P} \, \subseteq \, \irrepg{\mu} \, .$$ 


\end{remark}

\begin{definition} For any $\mu \in D$ we define a filtration $\mathcal{H}$ on $\Lhi{\lambda}{\mu}{P}$ as follows. Set $$\irrepfilt{\lambda}{\mu}{P}{n} \, := \, \varphi \: \Bigg( Hom_G \bigg(   \irrepg{\mu}^*   , \,   \cohomLfilt{0}{\lambda}{n \:}  \, \bigg) \Bigg) \subseteq \Lhi{\lambda}{\mu}{P} \, .$$
\end{definition}

\begin{theorem} \label{th:brygenfilt} Choose $\mu \in D$. We have $$\irrepfilt{\lambda}{\mu}{P}{n} = \irrepbrygen{\lambda}{\mu}{P}{n}{X}$$ for all $n \geq 0$.
\end{theorem}

\begin{proof} We give a sketch of the proof; the details follow from arguments similar to those in the proofs of Lemmas 5.6 and 5.7 and Proposition 5.9 of \cite{Bry89}.

Choose $\bar{f} \in \irrepfilt{\lambda}{\mu}{P}{n} \setminus \, \irrepfilt{\lambda}{\mu}{P}{n - 1}$. We need to check that $X^n . \bar{f} \neq 0$ and $X^{n + 1}.  \bar{f} = 0$. Set $$f := \varphi^{-1} \bar{f} \in Hom_G \bigg(   \irrepg{\mu}^*   , \,   \cohomLfilt{0}{\lambda}{n \:}  \, \bigg) \, .$$

For $v \in \irrepgd{\mu}$ set $$f_v := f(v) \in  \cohomL{0}{\lambda}  \, .$$ Choose $v \in \irrepgd{\mu}$ such that $$f_v \in \cohomLfilt{0}{\lambda}{n \:} \, \setminus \, \cohomLfilt{0}{\lambda}{n-1 \:} $$ and set $$N(t) := H + tX \, ;$$ this is a line in the fiber of $G/L$ over $eP$. For $g \in G$ we may interpret $f_v (gP)$ as a regular function on the fiber of $G/L$ over $gP$.

By an argument similar to that of Lemma 5.7 in \cite{Bry89}, one checks that there exists $g \in G$ such that the degree of the polynomial $t \mapsto f_v (gP) \Big(  g. \big(  N(t)  \big)  \Big)$ is $n$. By $G$-equivariance this implies that the degree of the polynomial $t \mapsto f_{gv}(eP)(N(t))$ is $n$.

By an analog of Lemma 5.6 in \cite{Bry89}, we have that the degree of the polynomial $t \mapsto f_{gv}(N(t))$ is the same as the degree of the polynomial $$t \mapsto \big(\textrm{exp}(-tX).f_{gv} \big)(eP)(H) \, .$$ In particular, the degree of this polynomial is $n$. As the fiber of $G/L$ over $eP$ is isomorphic to $\irrepd{\lambda}{P}$ we obtain that $n$ is the maximum of the degrees of the $\irrepd{\lambda}{P}$-valued polynomials $\big(\textrm{exp}(-tX).f_{v} \big)(eP)(H)$ as $v$ runs over the elements of $\irrepgd{\mu}$. By an argument similar to that of Proposition 5.9 of \cite{Bry89}, this implies that the degree of the polynomial $  \textrm{exp}(tX). \big( \varphi(f)  \big) \, $ is also $n$. Since \begin{eqnarray*}  \textrm{exp}(tX). \big( \varphi(f) \big) & = & (1 \, + \, tX \, + \, t^2X^2/2 \, + \, \ldots \,) \, . \,\varphi(f)  \\
& = & (1 \, + \, tX \, + \, t^2X^2/2 \, +\, \ldots \,) \, . \, \bar{f}
\end{eqnarray*} we see that $X^n . \bar{f} \neq 0$ and $X^{m}.  \bar{f} = 0$ for all $m > n$, as desired.

\end{proof}





\begin{remark}\label{rem:pr} Let $pr: G/B \rightarrow G/P$ be the projection map. Let $M$ be a $B$-module and let $\eb{M}{P/B}$ be the sheaf of sections of the $P$-equivariant bundle $P \times^B M$ on $P/B$. Then, by Proposition III.8.1 of \cite{Ha}, we have $$R^i pr_* \, \eb{M}{B} \congsp \EB{  \Cohom{i}{P/B}{  \eb{M}{P/B}  }  }{P} \, .$$
\end{remark}

\begin{lemma} \label{lem:cotsheaf} For each $i \geq 0$ there are isomorphisms $$\COhom{i}{ G/B }{ \eB{ \symmp{p} \otimes \C_{ -\lambda } }{B} } \congsp \COhom{i}{ G/P }{   \shcotpush{ \lambda }{ P }   } \congsp \Cohom{i}{ \pcot{P} }{ p^* \eb{\lambda}{P} }$$ of $G$-modules. Furthermore, the first isomorphism is an isomorphism of graded $G$-modules, where the grading on the first module is obtained from the grading on $\symmp{p}$ and the grading on the second is obtained from the grading on $\shcotpush{ \lambda }{ P }$.
\end{lemma}

\begin{proof} By Remark \ref{rem:function interpretation} we have that $$\shcotpush{\lambda}{P} \congsp \pcotpp{\lambda} \, .$$ Let $pr: G/B \rightarrow G/P$ be the projection map. Then $$\eB{ \symmp{p} }{B} \congsp pr^* \eB{ \symmp{p} }{P} \, ,$$ since both are $G$-equivariant bundles on $G/B$ with fiber $\symmp{p}$. By Remark \ref{rem:pr}, we have that $$R^i pr_* \, \eb{  \lambda  }{B} \, = \, R^i pr_* \, \eb{  \C_{-\lambda}  }{B} \congsp \EB{  \Cohom{i}{P/B}{  \eb{  \C_{-\lambda}  }{P/B}  }  }{P} \, .$$

Since $P/B \congsp L/B_L$, we have $$\Cohom{0}{P/B}{  \eb{  \C_{-\lambda}  }{P/B}  } \congsp \Cohom0{L/B_L}{  \eb{  \C_{-\lambda}  }{L/B_L}  } \congsp \irrepd{\lambda}{P}$$ and $$\Cohom{i}{L/B_L}{  \eb{  \C_{-\lambda}  }{L/B_L}  } = 0 $$ for $i > 0$. Thus $$pr_* \, \eb{  \lambda  }{B} \congsp \eb{  \lambda  }{P}$$ and $$R^i pr_* \, \eb{  \lambda  }{B} = 0$$ for $i > 0$.

The Leray spectral sequence and the projection formula now imply the first isomorphism (note that all isomorphisms here preserve the grading). The second isomorphism follows from the Leray spectral sequence again, as $p$ is an affine morphism.

\end{proof}

Let $X$ denote the group ring of $\Lambda$ over $\mb{Z}$ written multiplicatively; ie, $X$ is the $\mb{Z}$-algebra with generators $\{    e^\lambda : \lambda \in \Lambda    \}$ and relations $e^{ \lambda_1 }  e^{ \lambda_2 } = e^{ \lambda_1 + \lambda_2 }$. For any finite-dimensional $B$-module $M$ set $$  \textrm{ch} \, M := \sum_{ \substack{ \textrm{weights} \\ \gamma \textrm{ of} \: M }} e^\gamma  \, ,$$ where the weights are summed with multiplicity. Also set $$\chi(M) \, := \, \sum_{i \geq 0} (-1)^i \, \textrm{ch} \, \Cohom{i}{G/B}{  \eb{M}{B}  } \in X$$ and set $$\chichar{M} \, := \, \sum_{ \substack{ \textrm{weights} \\ \gamma \textrm{ of} \: M }} \, \sum_{i \geq 0} (-1)^i \, \textrm{ch} \, \Cohom{i}{G/B}{  \eb{  \chara{\gamma}  }{B}  } \in X \, ,$$ where the weights are summed with multiplicity. By the additivity of Euler characteristic we have $\chi(M) = \chichar{M}$ for all finite-dimensional $B$-modules $M$.



\begin{proposition} \label{lem:charcomp} Assume that $$ \Cohom{i}{  \pcot{P}  }{  p^* \eb{\lambda}{P}  } = 0$$ for all $i > 0$. Then $$\sum_{n \geq 0} \textrm{dim } Hom_G \bigg(   \irrepgd{\mu} , \,  \COhom{0}{  G/P  }{  \shcotgr{\lambda}{P}{n}  }  \bigg) \, q^n = \qana{\mu}{\lambda} $$ for all $\mu \in D$.
\end{proposition}

\begin{proof} By Lemma \ref{lem:cotsheaf} we have isomorphisms $$\Cohom{i}{ \pcot{P} }{ p^* \eb{\lambda}{P} }  \congsp \cohom{i}{ G/P }{ \shcotpush{\lambda}{P} }  \congsp \COhom{i}{ G/B }{ \eB{ \symmp{p} \otimes \C_{ -\lambda } }{B} }  $$ of graded $G$-equivariant sheaves for all $i$. Thus we need to show that $$\sum_{n \geq 0} \textrm{dim }    \textrm{Hom}_G \Bigg(  \irrepgd{\mu}  ,  \fcohsymdegs{\lambda}{p}{0}{n}  \Bigg) \,  q^n \, = \, m_\mu^{P, \lambda} (q) \, .$$

By the vanishing assumption, $$  \COhom{i}{ G/B }{ \eB{ \symmp{p} \otimes \C_{ -\lambda } }{B} }  = 0  $$ for all $i > 0$. Hence $$\textrm{ch} \left[ \fcohsymdegs{\lambda}{p}{0}{n} \right] \, = \, \chi \big(  \symmpg{p}{n} \otimes \chara{-\lambda}  \big)$$ for all $n \geq 0$.

For any $n \geq 0$ let $\pkoscoeff{n}$ be the degree-$n$ coefficient of the polynomial $\pkos$ and let $\qanacoeff{\mu}{\lambda}{n}$ be the degree-$n$ coefficient of $\qana{\mu}{\lambda}$ (recall the definition of $\pkos$ from section \ref{sec:gendefs}). For all $n \geq 0$ a standard computation now gives $$  \chi \big(  \symmpg{p}{n} \otimes \chara{-\lambda}  \big) \, = \,  \sum_{\gamma \in D} \qanacoeff{\gamma}{\lambda}{n} \chi \big( \chara{-\gamma} \big) \, = \,  \sum_{\gamma \in D} \qanacoeff{\gamma}{\lambda}{n} \textrm{ch} \big(\irrepgd{\gamma}\big)   \, .$$

That is, for all $\gamma \in D$ the multiplicity of $\irrepgd{\gamma}$ in $$\fcohsymdegs{\lambda}{p}{0}{n}$$ is $\qanacoeff{\gamma}{\lambda}{n}$, as desired.

\end{proof}

\bigskip
We can now prove Theorem \ref{th:brygen}.


\begin{proof} (of Theorem \ref{th:brygen}) Recall that we want to show the following: If $$\Cohom{  i  }{  \pcot{P}  }{  p^* \eb{\lambda}{P}  } = 0$$ for all $i >0$, then $$r_\mu^{X, \lambda} (q) \, = \, m_\mu^{P, \lambda} (q) \, .$$

First note that, by Theorem \ref{th:brygenfilt}, we have $$r_\mu^{X, \lambda} (q) \, = \, \sum_{n \geq 0} \textrm{dim} \left ( \frac{   \irrepfilt{\lambda}{\mu}{P}{n}   }{   \irrepfilt{\lambda}{\mu}{P}{n-1}   } \right) q^n \, . $$ By Proposition \ref{lem:charcomp} it now suffices to show that 
\begin{equation*} \label{eq:filt} \textrm{dim} \left ( \frac{   \irrepfilt{\lambda}{\mu}{P}{n}   }{   \irrepfilt{\lambda}{\mu}{P}{n-1}   } \right) \, = \, \textrm{dim } Hom_G \bigg(   \irrepgd{\mu} , \,  \COhom{0}{  G/P  }{  \shcotgr{\lambda}{P}{n}  }  \bigg)
\end{equation*}
for all $n \geq 0$ and $\mu \in D$.

Fix $\mu \in D$. Recall that $$\irrepfilt{\lambda}{\mu}{P}{n} \, = \, \varphi \: \bigg( Hom_G \Big(   \irrepg{\mu}^*   , \,   \cohomLfilt{0}{\lambda}{n \:}  \, \Big) \bigg) \subseteq \Lhi{\lambda}{\mu}{P} \, .$$ Using that $\varphi$ is an isomorphism and that the functor $$\textrm{Hom}_G \, (  \irrepgd{\mu}  , \, \bullet \,  )$$ is exact,  we have
\begin{eqnarray*} \textrm{dim} \left ( \frac{   \irrepfilt{\lambda}{\mu}{P}{n}   }{   \irrepfilt{\lambda}{\mu}{P}{n-1}   } \right) & = & \textrm{dim}  \left(  \frac{ Hom_G \Big(   \irrepg{\mu}^*   , \,   \cohomLfilt{0}{\lambda}{n \:}  \, \Big) }{ Hom_G \Big(   \irrepg{\mu}^*   , \,   \cohomLfilt{0}{\lambda}{n-1 \:}  \, \Big) }    \right) \\
   & = & \textrm{dim } Hom_G \left(   \irrepg{\mu}^*   , \,   \frac{  \cohomLfilt{0}{\lambda}{n \:}  }{  \cohomLfilt{0}{\lambda}{n-1 \:}  }  \, \right)  \, .
\end{eqnarray*}

Finally, note that by Proposition \ref{pr:grL} and the cohomology vanishing assumption, an easy induction shows that for all $n \geq 0$ there is an isomorphism
\begin{equation*} \label{eq:isompq} \frac{  \cohomLfilt{0}{\lambda}{n \:}  }{  \cohomLfilt{0}{\lambda}{n-1 \:}  } \congsp \COhom{0}{  G/P  }{  \shcotgr{\lambda}{P}{n}  }
\end{equation*} of $G$-modules. This completes the proof.


\end{proof}

\section{Cohomology of Flag Varieties} \label{sec:cohom}
\subsection{Overview and Background} \label{sec:vanintro}

\subsubsection{Definitions}

Let $P$ be a standard parabolic subgroup of $G$ and let $L$ be the Levi subgroup of $P$ containing $T$. Recall that $\Delta_P$, $\Delta^+_P$, and $\pi_P$ denote the roots, the positive roots, and the simple roots, respectively, of $L$.  Let $\rho_P$ denote the half-sum of all positive roots of $L$. Let $W_P \subseteq W$ denote the Weyl group of $L$. 

Recall also that for $\lambda \in \Lambda$ we say that $\lambda$ is $P$-dominant (resp. $P$-regular dominant) if if $\pair{\lambda}{\alpha} \geq 0$ (resp. $\pair{\lambda}{\alpha} > 0$) for all $\alpha \in \pi_P$. For $P$-dominant $\lambda$ we have the irreducible $P$-module $\irrep{\lambda}{P}$ of highest weight $\lambda$.

In this section we assume without loss of generality that all parabolics are standard.
%

\subsubsection{ } We begin by providing an outline of the results in this section. For any parabolic $P \subseteq G$ and $\lambda \in \Lambda$ let $r_P : G \times^B \nil{p} \rightarrow G/B$ be the bundle map and consider the cohomology groups $$\pc{i}{\lambda}{p} := \Cohom{i}{  G \times^B \nil{p}  }{  r_P^* \, \eb{\lambda}{B}  } \, .$$ By Remark \ref{rem:function interpretation} and Lemma \ref{lem:cotsheaf}, when $\lambda$ is $P$-dominant we have $$\pc{i}{\lambda}{p} \congsp \pcotcoh{\lambda}{P} \, ,$$ where $p_P : \pcot{P} \rightarrow G/P$ is the bundle map.

Broer \cite{Bro94} showed that these cohomology groups vanish for $i > 0$ in the following cases (see Theorem \ref{th:broer} and Propositions \ref{pr:broermin} and \ref{pr:broer} below):
\begin{itemize}
\item $P = B$ and $\cht{\lambda} = 0$ (see Lemma \ref{lem:cht0}(i) below for a characterization of the weights for which cht $= 0$). In particular, we have higher cohomology vanishing when $P = B$ and $\lambda \in D$.
\item $P$ is a minimal parabolic corresponding to a short simple root and $\lambda \in D$.
\item $P$ is any parabolic and $\lambda$ is any weight such that $\pair{\lambda}{\alpha} = -1$ for some $\alpha \in \pi_P$ (in fact, we have cohomology vanishing for all $i$ in this case).
\item $P$ is any parabolic and $\chara{\lambda}$ is a 1-dimensional $P$-module, i.e. $\pair{\lambda}{\alpha} = 0$ for all $\alpha \in \pi_P$.
\end{itemize}

In this chapter we show that these cohomology groups vanish for $i > 0$ in the following additional cases: 
\begin{itemize}
\item $P$ is any minimal parabolic and $\lambda \in D$ (section \ref{sec:minparvan}).
\item $P$ is any parabolic and $\lambda = \mu - 2 \rho_P$ for some $P$-regular dominant $\mu \in D$ (section \ref{subsec:P-reg dominant vanishing}).
\item $P$ is any parabolic in type $A$ and $\lambda \in D$ is regular (section \ref{sec:frobenius splitting}).
\end{itemize}

In section \ref{subsec:P-reg dominant vanishing} we also collect some corollaries that give vanishing for some higher cohomologies in certain cases.

\subsubsection{Background} \label{sec:rhovan} \quad 

\begin{definition} \label{def:cht} For $\gamma \in \Lambda$ let $\gamma^+$ denote the unique dominant weight in the Weyl group orbit of $\gamma$. By \cite{Bro94}, section 3, there is a unique dominant weight $\gamma^\star$ such that (a) $\gamma^\star \geq \gamma$ and (b) $\gamma^\star \leq \mu$ for all $\mu \in D$ such that $\mu \geq \gamma$. The \textbf{combinatorial height} of $\lambda$ is given by $$\hspace{-.45in} \cht{\lambda} \, := \, \textrm{max} \left\{ m : \textrm{there is a chain of dominant weights } \gamma^\star =: \gamma_0 < \gamma_1 < \cdots < \gamma_m := \gamma^+ \right\}.$$
\end{definition}

The following is a more natural way of classifying the weights of combinatorial height 0. Part (i) is due to Broer (\cite{Bro97}, Proposition 2), and (ii) follows readily from (i).

\begin{lemma} \label{lem:cht0} \quad
	\begin{enumerate} [(i)]
	\item Let $\lambda \in \Lambda$. Then $\cht{\lambda} = 0$ iff $\lambda(\beta^\vee) \geq -1$ for all $\beta \in \Delta^+$. In particular, $\cht{\lambda} = 0$ for all $\lambda \in D$.
	\item Let $\lambda \in \Lambda$ with $\cht{\lambda} = 0$ and let $\mu \in D$. Then $\cht{\lambda + \mu} = 0$.
	\end{enumerate}
\end{lemma}

\begin{theorem} \label{th:broer} (Broer, \cite{Bro94} Theorem 3.9) 
	\begin{enumerate}[(i)]
	\item Let $P \subseteq G$ be any parabolic subgroup. Then $\pc{i}{\lambda}{p} = 0$ for all $i > 0$ and $\lambda \in D$ such that $\chara{\lambda}$ is a $P$-module (i.e., $\pair{\lambda}{\alpha} = 0$ for all $\alpha \in \pi_P$). 
	\item Choose $\gamma \in \Lambda$. Then $\cht{\gamma} = 0$ iff $\fc{i}{\gamma} = 0$ for all $i > 0$.
	\item Choose $\gamma \in \Lambda$. Then $\fc{i}{\gamma} = 0$ for all $i > \cht{\gamma}$.

	\end{enumerate}

\end{theorem}

\bigskip

\subsection{Vanishing for minimal parabolics}
\subsubsection{Some Combinatorics} \label{subsec:combinatorics} 

We collect a few combinatorial results.

\begin{proposition} \label{pr:comb}(Thomsen, \cite{T00} Propositions 1 and 3 and Corollary 2) Let $\lambda \in \Lambda$, $\alpha \in \pi$, and $\beta \in \Delta^+$.
	\begin{enumerate}[(i)]
	\item If $\pair{\lambda}{\beta} \geq 0$ then $\lambda^+ < (\lambda + \beta)^+$.
	\item If $\pair{\lambda}{\beta} = -1$ then $\lambda^+ = (\lambda + \beta)^+$.
	\item If $\pair{\lambda}{\beta} < -1$ then $\lambda ^+ > (\lambda + \beta)^+$.
	\item If $\pair{\lambda}{\alpha} < 0$ then $\lambda^\star = (\lambda + \alpha)^\star$.
	\item If $\pair{\lambda}{\alpha} = -1$ then $\cht{\lambda} = \cht{\lambda + \alpha}$.
	\item If $\pair{\lambda}{\alpha} \leq -2$ then $\cht{\lambda} > \cht{\lambda + \alpha}$.
	\item If $\pair{\lambda}{\alpha} \leq 0$ then $\cht{\lambda} \geq \cht{s_\alpha \lambda}$.
	\item If $\pair{\lambda}{\alpha} \leq -2$ then  $\cht{\lambda} > \cht{s_\alpha * \lambda}$.
	\end{enumerate}
\end{proposition}




\bigskip

\begin{lemma} \label{cor:shortroot} $\cht{\beta + \mu} = 0$ for all short roots $\beta$ and $\mu \in D$. 
\end{lemma}

\begin{proof} This is immediate from Lemma \ref{lem:cht0} and the fact that $\pair{\beta}{\gamma} \geq -1$ for all $\gamma \in \Delta^+$.
\end{proof}

\begin{lemma} \label{lem:chtsimple} (Broer, \cite{Bro94}) Assume that $G$ is simple. Let $\alpha_1, \, \alpha_2, \, \ldots , \, \alpha_k$ be any (nonempty) collection of short simple orthogonal roots. Then $$\cht{  \sum_{i=1}^k \alpha_i  } = k - 1 \, .$$
\end{lemma}

Let $\inner{ \cdot }{ \cdot }$ denote a Weyl group-invariant inner product on $\mf{h}$, normalized so that $\normsq{\alpha} = 1$ for all short simple roots $\alpha$. Recall that $\rootlat$ denotes the root lattice of $T$.

\begin{lemma} \label{lem:innerprod} \quad
	\begin{enumerate}[(i)]
	\item $\normsq{\lambda} \in \mathbb{Z}^+$ for all $\lambda \in \rootlat$.
	\item Let $\mu$ and $\lambda$ in $D$. If $\mu < \lambda$ then $\normsq{\mu} < \normsq{\lambda}$.
	\end{enumerate}
\end{lemma}

\begin{lemma} \label{lem:chtcomp} Let $\lambda \in \Lambda$.
	\begin{enumerate}[(i)]
	\item $\cht{\lambda} \leq \normsq{\lambda^+} - \normsq{\lambda^\star} \, = \, \normsq{\lambda} - \normsq{\lambda^\star}$.
	\item For any $\gamma \in \Lambda$ such that (1) $\lambda \leq \gamma \leq \lambda^\star$ and (2) $\cht{\gamma} = 0$, we have $$\cht{\lambda} \leq \normsq{\lambda} - \normsq{\gamma}.$$
	\end{enumerate}
\end{lemma}

\begin{proof} (i) For any $\mu$, $\gamma \in D$ with $\mu < \gamma$, we have
$$ \normsq{\gamma} - \normsq{\mu}  \, = \,  \normsq{\gamma - \mu} + 2 \inner{\gamma - \mu}{\mu} \, .$$
Now,  $\gamma - \mu \in \rootlat$ implies $\normsq{\gamma - \mu} \in \mathbb{Z}^+$ by \ref{lem:innerprod} (i); and $\inner{\gamma - \mu}{\mu} \geq 0$ since $\mu \in D$ and $\gamma - \mu > 0$. Thus $$\normsq{\gamma - \mu} + 2 \inner{\gamma - \mu}{\mu} > 0 \, .$$

The result now follows from the definition of cht.

(ii) By the definition of $\lambda^\star$ and the fact that $\lambda \leq \gamma \leq \gamma^\star$, we have $\gamma^\star \geq \lambda^\star$. By the definition of $\gamma^\star$ and the fact that $\gamma \leq \lambda^\star$ we have $\gamma^\star \leq \lambda^\star$. Thus $\lambda^\star = \gamma^\star$. And, since $\cht{\gamma} = 0$, we have $\gamma^+ = \gamma^\star$. Hence $\gamma^+ = \lambda^\star$. The result now follows from the Weyl-group invariance of $\inner{\,}{\,} \,$.

\end{proof}

\subsubsection{Vanishing Theorem} \label{sec:minparvan} 

\quad

\quad

\begin{remark} \label{rem:shortexact} Let $P$ be a minimal parabolic corresponding to a simple root $\alpha$. We have a short exact sequence of $B$-modules $$\shortexact{  \nil{p}  }{  \mf{n}  }{  \chara{\alpha}  }$$ which gives rise (upon taking the dual) to a short exact Koszul sequence $$\shortexact{  \symm \otimes \chara{-\alpha}  }{  \symm   }{  \symmp{p}   } \, .$$ 

Let $\mu$ be any weight. Tensoring with $\chara{-\mu}$ we obtain the short exact sequence $$\shortexact{  \symm \otimes \chara{-\alpha - \mu}  }{  \symm \otimes \chara{-\mu}  }{  \symmp{p} \otimes \chara{-\mu}  } $$ of $B$-modules and hence the short exact sequence $$\shortexact{  \eB{\symm}{B} \otimes \eb{\alpha + \mu}{B}  }{  \eB{\symm}{B} \otimes \eb{\mu}{B}  }{  \eB{\symmp{p}}{B} \otimes \eb{\mu}{B}  }$$ of sheaves on $G/B$. This short exact sequence will be very useful in the sequel.

\end{remark}

The following proposition is implicit in \cite{Bro94}. We provide a proof here for completeness. 
\begin{proposition} \label{pr:broermin} Let $G$ be simple. Let $P$ be a minimal parabolic corresponding to a short simple root $\alpha$. Let $\mu \in D$. Then $\pc{i}{\mu}{p} = 0$ for all $i > 0$.
\end{proposition}

\begin{proof}
Using the long exact sequence of cohomology associated to the short exact sequence of sheaves in Remark \ref{rem:shortexact} above, the result now follows from Lemma \ref{lem:cht0} (ii), Theorem \ref{th:broer} (ii), and Lemma \ref{cor:shortroot}.
\end{proof}

The following result is a reformulation of Broer's Lemma 3.1 in \cite{Bro94}.








\begin{proposition} \label{pr:broer} Let $P$ be a standard parabolic and choose $\lambda \in \Lambda$. If there is $w \in W_P$ such that $w*\lambda$ is $P$-dominant then $$\pc{i}{\lambda}{p} = \pc{i - l(w)}{w*\lambda}{p} \congsp \Cohom{  i - l(w)  }{  \pcot{P}  }{  p_P^* \, \eb{ w*\lambda }{P}  }$$ for all $i \geq 0$. In particular, if $i < l(w)$ then $$\pc{i}{\lambda}{p} = 0 \, .$$ Further, if there does not exist $w \in W_P$ such that $w*\lambda$ is $P$-dominant, then $$\pc{i}{\lambda}{p} = 0$$ for all $i$.

\end{proposition}

\begin{remark} Choose $\lambda \in \Lambda$. Then there does not exist $w \in W_P$ such that $w*\lambda$ is $P$-dominant if and only if $\pair{\lambda}{\beta} = - \pair{\rho_P}{\beta}$ for some $\beta \in \Delta^+_P$. In particular, if $\pair{\lambda}{\alpha} = -1$ for some $\alpha \in \pi_P$ then by Proposition \ref{pr:broer}, $\pc{i}{\lambda}{p} = 0$ for all $i$.

\end{remark}

We now have the following theorem.

\begin{theorem} \label{th:minpar} Let $P$ be any minimal parabolic. Let $\mu \in D$. Then $\pc{i}{\mu}{p} = 0$ for all $i > 0$.
\end{theorem}

\begin{proof} By Proposition \ref{pr:broermin}, we need to check this result for parabolics corresponding to long simple roots. Let $\alpha$ be a long simple root and let $P$ be the standard parabolic defined by $\alpha$.  By Theorem \ref{th:broer} and the long exact sequence in cohomology associated to the short exact sequence $$\shortexact{  \eB{\symm}{B} \otimes \eb{\alpha + \mu}{B}  }{  \eB{\symm}{B} \otimes \eb{\mu}{B}  }{  \eB{\symmp{p}}{B} \otimes \eb{\mu}{B}  }$$ of sheaves on $G/B$ (cf Remark \ref{rem:shortexact} above), to prove the theorem for $\mf{p}$ it suffices to show that $$\cht{\alpha + \mu} \leq 1$$ for all $\mu \in D$. 

We now check the cht condition case by case. We shall label the simple roots as $\pi = \left\{     \alpha_1, \ldots, \alpha_n    \right\}$.



\bigskip

(1) Type $B_n$.

We use the standard labelling for the simple roots, where $\alpha_n$ is the unique short simple root. Let $\mu \in D$ and let $\alpha$ be any long simple root. If $\pair{\mu}{\alpha_n} > 0$ then $\pairp{\mu + \alpha}{\beta} \geq -1$ for all $\beta \in \Delta^+$, and hence $\cht{\mu + \alpha} = 0$. Thus we may assume that $\pair{\mu}{\alpha_n} = 0$. We may also assume that $\mu \neq 0$ as the higher cohomology vanishing when $\mu = 0$ follows from Theorem \ref{th:broer} (i).

Set $$m := \textrm{max} \, \{i : \pair{\mu}{\alpha_i} > 0   \} \, < \, n \, .$$ We claim that $\cht{\mu + \alpha_l} = 0$ for all $0 < l < m$ (if $m = 1$ this condition is of course empty and there is nothing to check). Fix $l$ with $1 \leq l < m$. If $\pair{\mu}{\alpha_j} > 0$ for some $j < l$ then set $$c_1 :=  \textrm{max} \, \{j : j < l \textrm{ and } \pair{\mu}{\alpha_j} > 0   \} \, ;$$ if $\pair{\mu}{\alpha_j} = 0$ for all $j < l$ (or if $l = 1$) then set $c_1 = 0$. Also set $$c_2 := \textrm{min} \, \{k : l < k \textrm{ and } \pair{\mu}{\alpha_k} > 0   \} \leq m \, .$$ Set $$\displaystyle \gamma := \sum_{i = c_1 + 1}^{c_2 - 1} \mu + \alpha_i \, . $$ By Proposition \ref{pr:comb} (v) and an easy induction on $c_2 - c_1 - 1$ we have $\cht{\mu + \alpha_l} = \cht{\gamma}$. Furthermore, one checks easily that $\gamma \in D$. This shows the claim, and thus it suffices to consider $\mu + \alpha_k$ for $m \leq k < n$.

For any $k \in \mathbb{N}$ with $m \leq k < n$, set $$\beta_k := \sum_{i = k}^n \alpha_i \in \Delta^+ .$$ An easy induction on $n$ utilizing Proposition \ref{pr:comb} (iv) shows that $$(\mu + \alpha_k)^\star = (\mu + \beta_k)^\star \, ,$$ since $\pair{\mu}{\alpha_j} = 0$ for $j > m$.  Thus $$\alpha_k + \mu \, \leq \, \beta_k + \mu \, \leq \, (\alpha_k + \mu)^\star.$$ Further, $\beta_k$ is a short root, so $\cht{\beta_k + \mu} = 0$ by Lemma \ref{cor:shortroot}. So, by Lemma \ref{lem:chtcomp} and the fact that $   \inner{ \mu }{ \beta_k - \alpha_k } = 0   $, we have
\begin{eqnarray*}
\cht{\alpha_k + \mu} & \leq & \normsq{\alpha_k + \mu} - \normsq{\beta_k + \mu} \\
& = & \normsq{\alpha_k} - \normsq{\beta_k} \\
& = & 1 \, .
\end{eqnarray*} 

\bigskip

(2) Type $C_n$.

We use the standard labelling for the simple roots, where $\alpha_n$ denotes the unique long simple root. Let $\mu \in D$. If $\pair{\mu}{\alpha_{n-1}} > 0$ then $\pairp{\mu + \alpha_n}{ \beta } \geq -1$ for all $\beta \in \Delta^+$ and hence $\cht{\mu + \alpha_n} = 0$. Thus we may assume that $\pair{\mu}{\alpha_{n-1}} = 0$.


By Proposition \ref{pr:comb} (iv), since $\pairp{\mu + \alpha_n}{\alpha_{n-1}} < 0$, we get $(\alpha_{n-1} +\alpha_n + \mu)^\star = (\alpha_n + \mu)^\star$.  Thus $$\alpha_n + \mu \, \leq \, \alpha_{n-1} + \alpha_n + \mu \, \leq \, (\alpha_n + \mu)^\star.$$ Further, $\alpha_{n-1} + \alpha_n$ is a short root, so $\cht{\alpha_{n-1} + \alpha_n + \mu} = 0$ by Corollary \ref{cor:shortroot}. So, by Lemma \ref{lem:chtcomp}, we have
\begin{eqnarray*}
\cht{\alpha_n + \mu} & \leq & \normsq{\alpha_n + \mu} - \normsq{\alpha_{n-1} + \alpha_n + \mu} \\
& = & - \left[   \, \normsq{  \alpha_{n-1}  } + 2 \inner{\alpha_{n-1}}{\alpha_n}   \right] \\
& = & -(1 - 2) \\
& = & 1 \, .
\end{eqnarray*} 

\bigskip

(3) Type $G_2$.

We have a short simple root $\alpha_1$ and a long simple root $\alpha_2$.

Note that if $\pair{\mu}{\alpha_1} > 2$  then $\mu + \alpha_2 \in D$. If $\pair{\mu}{\alpha_1} = 2$ then $$\pairp{\mu + \alpha_2}{ \alpha_1 } = -1$$ and thus $$\cht{\mu + \alpha_2} = \cht{\mu + \alpha_1 + \alpha_2} = 0$$ by Proposition \ref{pr:comb} (v) (note that $\mu + \alpha_1 + \alpha_2 \in D$).

Hence we may assume that $\pair{\mu}{\alpha_1} \leq 1$. We now consider the cases $\pair{\mu}{\alpha_1} = 0$ and $\pair{\mu}{\alpha_1} = 1$ separately.

Assume first that $\pair{\mu}{\alpha_1} = 1$. Since $\alpha_1 + \alpha_2$ is a short root, by Corollary \ref{cor:shortroot} we get that $\cht{\alpha_1 + \alpha_2 + \mu} = 0$. By the assumption on $\mu$, we have $\pairp{ \mu + \alpha_2 }{ \alpha_1 } < 0$. Thus, by Proposition \ref{pr:comb} (iv), $(\alpha_2 + \mu)^\star = (\alpha_1 + \alpha_2 + \mu)^\star$. Hence $$\alpha_2 + \mu \, < \, \alpha_1 + \alpha_2 + \mu \, \leq \, (\alpha_2 + \mu)^\star$$ and by Lemma \ref{lem:chtcomp} (ii) we have 
\begin{eqnarray*}
\cht{\alpha_2 + \mu} & \leq & \normsq{\alpha_2 + \mu} - \normsq{\alpha_1 + \alpha_2 + \mu} \\
& = & - \left[   \normsq{\alpha_1} + 2 \inner{\mu}{\alpha_1} + 2 \inner{\alpha_1}{\alpha_2}   \right] \\
& = & -(1 + \pair{\mu}{\alpha_1} - 3) \\
& = & 1 \, .
\end{eqnarray*} 


Now assume that $\pair{\mu}{\alpha_1} = 0$. Let $\mf{p}$ be the minimal parabolic associated to $\alpha_2$. If $\mu = 0$ then $$\pc{i}{\mu}{p} = 0$$ for $i > 0$, by Theorem \ref{th:broer}.

If $\mu \neq 0$ then $\pair{\mu}{\alpha_2} > 0$ and $\cht{\mu + \alpha_1 + \alpha_2} = 0$ by Lemma \ref{cor:shortroot}, as $\alpha_1 + \alpha_2$ is a short root. Let $\mf{q}$ be the minimal parabolic corresponding to the short simple root $\alpha_1$. We have the short exact Koszul complex
\begin{eqnarray*}  0 & \rightarrow & \eB{\symm}{B} \otimes \eb{ \mu + \alpha_1 + \alpha_2 }{B} \rightarrow \eB{\symm}{B} \otimes \eb{ \mu + \alpha_2 }{B} \\
& \rightarrow & \eB{\symmp{q}}{B} \otimes \eb{ \mu + \alpha_2 }{B} \, \rightarrow \, 0
\end{eqnarray*} of sheaves on $G/B$. Since $\cht{\mu + \alpha_1 + \alpha_2} = 0$, by Theorem \ref{th:broer} (ii) we have $$\fc{i}{  \mu + \alpha_1 + \alpha_2  } = 0$$ for all $i > 0$.

Now, $\pairp{ \mu + \alpha_2 }{ \alpha_1 } = -3$. Thus $$r_{\alpha_1} * (\mu + \alpha_2) = \mu + 2\alpha_1 + \alpha_2 \in D \, .$$ By Propositions \ref{pr:broer} and \ref{th:minpar} we now have $$\pc{i+1}{\mu + \alpha_2}{q} \congsp \pc{i}{\mu + 2 \alpha_1 + \alpha_2}{q} = 0$$ for all $i > 0$. Thus, by the short exact sequence above, $$\fc{i}{\mu + \alpha_2} = 0$$ for all $i > 1$.

Now, from the short exact Koszul complex
\begin{eqnarray*}  0 & \rightarrow & \eB{\symm}{B} \otimes \eb{ \mu + \alpha_2 }{B} \rightarrow \eB{\symm}{B} \otimes \eb{ \mu }{B} \\
& \rightarrow & \eB{\symmp{p}}{B} \otimes \eb{ \mu }{B} \, \rightarrow \, 0
\end{eqnarray*} and the fact that $\mu \in D$ we obtain that $\pc{i}{\mu}{p} = 0$ for all $i > 0$, as desired.

\bigskip

(4) Type $F_4$.

We have 4 simple roots $\alpha_i$, $\, i = 1, 2, 3, 4$, where $\alpha_3$ and $\alpha_4$ are the long simple roots. Fix $\mu \in D$. We consider the cases $\pair{\mu}{\alpha_2} > 0$ and $\pair{\mu}{\alpha_2} = 0$ separately.

First assume that $\pair{\mu}{\alpha_2} > 0$. Note that $\mu + \alpha_3 + \alpha_4 \in D$. In the following we use Proposition \ref{pr:comb} (v) extensively. If $\pair{\mu}{\alpha_3} > 0$ and $\pair{\mu}{\alpha_4} > 0$ then $\mu + \alpha_3 \in D$ and $\mu + \alpha_4 \in D$. If $\pair{\mu}{\alpha_3} = 0$ and $\pair{\mu}{\alpha_4} > 0$ then $\mu + \alpha_3 \in D$ and $$\cht{\mu + \alpha_4} = \cht{  \mu + \alpha_3 + \alpha_4  } = 0 \, .$$ If $\pair{\mu}{\alpha_4} = 0$ and $\pair{\mu}{\alpha_3} > 0$ then $\mu + \alpha_4 \in D$ and $$\cht{\mu + \alpha_3} = \cht{  \mu + \alpha_3 + \alpha_4  } = 0 \, .$$ Finally, if $\pair{\mu}{\alpha_3} = \pair{\mu}{\alpha_4} = 0$ we have that $$\cht{\mu + \alpha_3} = \cht{\mu + \alpha_3 + \alpha_4} = \cht{\mu + \alpha_4} = 0 \, .$$ Thus, if $\pair{\mu}{\alpha_2} > 0$ we have $$\cht{\mu + \alpha_3} = \cht{\mu + \alpha_4} = 0 \, .$$

Now assume that $\pair{\mu}{\alpha_2} = 0$. Fix $k = 3$ or $4$; we now show that $\cht{\mu + \alpha_k} \leq 1$. Note that if $\pair{\mu}{\alpha_3} > 0$ then $\mu + \alpha_4 \in D$, so we may assume that $\pair{\mu}{\alpha_i} = 0$ for $1 < i < k$ (this condition is obviously already satisfied if $k=3$).  Set $$\beta_k := \sum_{i = 2}^k \alpha_i.$$ Then $\beta_k$ is a short root and $\cht{  \beta_k + \mu  } = 0$ by Lemma \ref{cor:shortroot}. Also, by our assumptions on $\mu$, $(\beta_k + \mu)^\star = (\alpha_k + \mu)^\star$ by Proposition \ref{pr:comb} (iv). Hence $$\alpha_k + \mu\, < \, \beta_k + \mu \leq (\alpha_k + \mu)^\star$$ and we have (using $\delta_{k, \, 4}$ to denote the terms that occur only when $k = 4$):

\begin{eqnarray*}
\cht{\alpha_k + \mu} & \leq & \normsq{\alpha_k + \mu} - \normsq{\beta_k + \mu} \\
& = & - \left[ \,   \normsq{\alpha_2} +  \delta_{k, \, 4} \normsq{\alpha_3} + 2 \inner{\alpha_2}{\alpha_3} + 2 \delta_{k, \, 4} \inner{\alpha_3}{\alpha_4}  \right] \\
& = & -\Big[   1 + 2 \delta_{k, \, 4} - 2 - 2 \delta_{k, \, 4}   \Big] \\
& = & 1 \, .
\end{eqnarray*}

\end{proof}

\bigskip

\subsection{ $\mu - 2 \rho_P$ vanishing } \label{subsec:P-reg dominant vanishing} \quad

We have the following formulation of the Grauert-Riemenschneider vanishing theorem due to Kempf \cite{Ke76}.

\begin{theorem} \label{th:gr} (Grauert-Riemenschneider) Let $X$ be a smooth complex variety and let $f : X \rightarrow Y$ be a proper morphism of complex varieties. Let $\omega_X$ denote the canonical bundle of $X$. If $n$ is the dimension of a generic fiber of $f$ then $R^if_*(\omega_X) = 0$ for all $i > n$.
\end{theorem}

The proof of the following theorem uses ideas from Broer's proof of his theorem 2.2 in \cite{Bro94} and appears here in a strengthened form due to a new proof from the referee.

\begin{theorem} \label{th:vanishing} Let $P$ be any standard parabolic and let $\mu \in D$ be $P$-regular dominant. Then $\pc{i}{ \mu - 2 \rho_P }{p} = 0$ for all $i > 0$. 

\end{theorem}

\begin{proof} Set $Y := G \times^B ( \nil{p} \times \chara{\mu} )$. The natural map $$P \times^B \chara{\mu} \rightarrow P.\chara{\mu} \subseteq \irrepgw{\mu}{\mu}$$ is birational, since the stabilizer in $P$ of $\irrepgw{\mu}{\mu}$ is $B$ by the $P$-regular dominance of $\mu$. From this and the isomorphism $$P \times^B ( \nil{p} \times P.\chara{\mu} ) \congsp \nil{p} \times ( P \times^B \chara{\mu} )$$ we now obtain a birational map \begin{eqnarray} \label{eq:generically finite}   Y & \cong & G \times^P \big( P \times^B ( \nil{p} \times \chara{\mu} ) \big) \nonumber \\ & \cong & G \times^P \big(  \nil{p} \times ( P \times^B \, \chara{\mu} )  \big) \, \rightarrow \, G \times^P \big( \nil{p} \times ( P.\chara{\mu} ) \big)   \, .\end{eqnarray}

Now, the natural map $$  G \times^P \nil{p} \rightarrow \mf{g}   $$ is generically finite; choose $X \in \mf{g}$ such that the fiber over $X$ is finite. Then the fiber over $(X, 0)$ in the natural map \begin{equation} \label{eq:generically finite 2} \hat{f} : G \times^P \big( \nil{p} \times ( P.\chara{\mu} ) \big) \, \rightarrow \, \mf{g} \times \irrepg{\mu}  \end{equation} is also finite, so $\hat{f}$ generically has finite fibers.

By (\ref{eq:generically finite}) and (\ref{eq:generically finite 2}) we now obtain a morphism $$ f : Y \, \rightarrow \, \mf{g} \times \irrepg{\mu} \, .$$ This map is proper, since it factors as $$  Y \, \hookrightarrow \, G / B \times \mf{g} \times \irrepg{\mu}  \, $$ followed by the projection onto the last two coordinates; and it is also generically finite, since the maps in (\ref{eq:generically finite}) and (\ref{eq:generically finite 2}) generically have finite fibers.

We now obtain from the Grauert-Riemenschneider theorem that $$\cohom{i}{Y}{\omega_Y} = 0$$ for all $i > 0$. By an argument similar to the proof of Theorem 2.2 in \cite{Bro94}, this implies that $$\COhom{i}{G/B}{  \eB{  \symmp{p} \otimes ( \chara{ -2 \rho_P } \otimes \chara{ \mu } )^*  }{B}  } = 0$$ for all $i > 0$. This completes the proof.
\end{proof}

We now collect a few corollaries of Theorem \ref{th:vanishing}.

\begin{corollary} \label{cor:longest} Let $P$ be any standard parabolic. Recall that $w_0^P$ is the longest element of the Weyl group $W_P \subseteq W$ corresponding to $\pi_P$. Choose $\mu \in \Lambda$ such that $w_0^P(\mu) \in D$ and $w_0^P$ is of minimal length with this property. Then $\pc{i}{\mu}{p} = 0$ for all $i \neq l(w_0^P)$.

\end{corollary}

\begin{proof} If $\pair{\mu}{\alpha} = -1$ for any $\alpha \in \pi_P$ then the result follows immediately from Proposition \ref{pr:broer}. Thus we may assume that $\pair{\mu}{\alpha} \leq -2$ for all $\alpha \in \pi_P$ and hence $\pairp{w_0^P \mu}{\alpha} \geq 2$ for all $\alpha \in \pi_P$. This implies that $$\pairp{w_0^P * \mu}{\alpha} = \pairp{ w_0^P \mu - 2 \rho_P }{\alpha} \geq 0$$ for all $\alpha \in \pi_P$. Since $w_0^P * \mu \in D$ this gives $$w_0^P * \mu + 2 \rho_P \, = \, w_0^P \mu \in D \, .$$ Thus $\pc{i}{w_0^P * \mu}{p} = 0$ for all $i > 0$ by Theorem \ref{th:vanishing}. The result now follows from Proposition \ref{pr:broer}.

\end{proof}

\begin{corollary} \label{cor:simple} Let $\mu \in \Lambda$ be such that $\mu \notin D$ and $r_\alpha (\mu) \in D$ for some $\alpha \in \pi$. Then $\fc{i}{\mu} = 0$ for all $i > 1$.
\end{corollary}

\begin{proof} Let $\mf{p}$ be the minimal parabolic with $\pi_P = \{\alpha\}$. Consider the Koszul complex $$\hspace{-.25in} \shortexact{  \eB{\symm}{B} \otimes \eb{\alpha + \mu}{B}  }{  \eB{\symm}{B} \otimes \eb{\mu}{B}  }{  \eB{\symmp{p}}{B} \otimes \eb{\mu}{B}  }$$ of Remark \ref{rem:shortexact}. We use induction on $-\mu(\alpha^\vee)$.

If $\mu(\alpha^\vee) = -1$ then $\pc{i}{\mu}{p} = 0$ for all $i$ by Proposition \ref{pr:broer}. Also, $\mu + \alpha = r_\alpha(\mu) \in D$, so $\fc{i}{\mu + \alpha} = 0$ for all $i > 0$ by Theorem \ref{th:broer} (ii) (this also follows from Theorem \ref{th:vanishing}). Thus, by the Koszul complex above, $\fc{i}{\mu} = 0$ for all $i > 0$.

If $\pair{\mu}{\alpha} = -2$ then $\pc{i}{\mu}{p} = 0$ for all $i > 1$ by Corollary \ref{cor:longest}. Further, as $r_\alpha(\mu) \in D$ and $\pairp{\mu + \alpha}{\alpha} = \pairp{r_\alpha \mu - \alpha}{\alpha} = 0$, we have that $\mu + \alpha \in D$. The result now follows from the Koszul complex above.

Now assume that $\mu(\alpha^\vee) = -n$ for some $n > 2$. By induction assume that $\fc{i}{\lambda} = 0$ for all $i > 1$ and $\lambda$ such that $r_\alpha (\lambda) \in D$ and $0 < -\lambda(\alpha^\vee) < n$.

Note that $$r_\alpha (\mu + \alpha) = r_\alpha(\mu) - \alpha \in D,$$ since $$(r_\alpha(\mu) - \alpha) (\alpha^\vee) = n - 2 \geq 0$$ and $$(r_\alpha(\mu) - \alpha)(\gamma^\vee) \geq r_\alpha(\gamma^\vee) \geq 0$$ for any $\gamma \in \pi \setminus \{\alpha\}$. Thus, by the induction hypothesis, $\fc{i}{\mu + \alpha} = 0$ for all $i > 1$. Furthermore, $\pc{i}{\mu}{p} = 0$ for all $i > 1$ by Corollary \ref{cor:longest}. The result now follows easily from the Koszul complex above.
\end{proof}

\begin{remark} Let $\mu \in \Lambda$ and $\alpha \in \pi$ be such that $\mu \notin D$ and $r_\alpha ( \mu ) \in D$. Let $\mf{p}$ be the minimal parabolic corresponding to $\alpha$. By Corollary \ref{cor:simple}, we now obtain a short exact sequence $$   \shortexact{  \fc{1}{ \mu+ \alpha }  }{  \fc{1}{ \mu }  }{  \pc{1}{\mu}{p}  }   \, .$$ Furthermore, if $\pair{\mu}{\alpha} = -1$, then all of these cohomology groups are 0.
\end{remark}
\subsection{Vanishing in Type $A$} \label{sec:frobenius splitting} \quad


Fix a prime $p$; for the rest of this section we assume that all schemes are $\algclosed{p}$-schemes unless otherwise specified, where $\mb{F}_p$ is the finite field with $p$ elements and $\algclosed{p}$ is its algebraic closure. The main reference for this section is \cite{BK}.

\subsubsection{Frobenius splitting}

Let $X$ be a scheme over $\algclosed{p}$. We define a morphism $F_X$ of schemes over $\mb{F}_p$ as follows. Set $F_X(x) = x$ for all $x \in X$ and define $F_X^\# : \struct{X} \rightarrow F_{X*} \, \struct{X}$ to be the $p^{\textrm{th}}$ power map $f \mapsto f^p$; this is clearly an $\mb{F}_p$-linear map. Note that $F_X$ is not a morphism of schemes over $\algclosed{p}$. This morphism is called the \textbf{absolute Frobenius morphism}. Generally when the context is clear we'll drop the subscript and just write $F$.


\begin{definition} We say that $X$ is \textbf{Frobenius split} if there is an $\struct{X}$-linear map $\varphi : F_* \struct{X} \rightarrow \struct{X}$ such that $\varphi \circ F^\#$ is the identity map on $\struct{X}$.
\end{definition}

The following is the essential result we require from the theory of Frobenius splitting.

\begin{proposition} \label{pr:frobvanishing} (\cite{BK}, Lemma 1.2.7) Let $X$ be a Frobenius split scheme and let $\mc{L}$ be an invertible sheaf on $X$. Then for all $i \geq 0$ there is an injection $$\cohom{i}{X}{  \mc{L}  } \hookrightarrow \cohom{i}{X}{  \mc{L}^p  }$$ (as $ \mb{F}_p $-vector spaces). In particular, if $\cohom{i}{X}{  \mc{L}^n  } = 0$ for all $n \gg 0$ then $\cohom{i}{X}{  \mc{L}  } = 0$ also.
\end{proposition}

\begin{corollary} (\cite{BK}, Theorem 1.2.8) \label{cor:frobvanishing} Assume that $X$ is Frobenius split and that there is a proper morphism from $X$ to an affine variety. Let $\mc{L}$ be an ample invertible sheaf on $X$. Then $\cohom{i}{X}{\mc{L}} = 0$ for all $i > 0$.
\end{corollary}

\subsubsection{Application to cohomology of cotangent bundles of flag varieties}

We now apply Frobenius splitting methods to flag varieties in type $A$.



The following theorem is a result due to van der Kallen and is a strengthening of Theorem 3.8 in \cite{MV}.

\begin{theorem} \label{th:BK} (van der Kallen, \cite{vdK}) Let $p$ be any prime. Set $\mc{G} := SL_n \big( \algclosed{p} \big)$ and let $\mc{X}$ be the full flag variety $\mc{G}/\mc{B}$ of $\mc{G}$. Then the bundles $\mc{G} \times^\mc{B} (\nil{p})_{\bar{p}}$ on $\mc{X}$ are Frobenius split, where $(\nil{p})_{\bar{p}}$ is the nilradical of any standard parabolic subgroup $\mf{p}_{\bar{p}}$ of Lie$(\mc{G})$.
\end{theorem}

We now come to the main result, which is a modification of Theorem 5.2.11 in \cite{BK} (see also \cite{KLT}).

\begin{theorem} Let $G$ be a semisimple algebraic group over $\C$ with all components of type $A$. Let $P$ be a standard parabolic subgroup of $G$. Then $$\pc{i}{\lambda}{p} = 0$$ for all regular $\lambda \in D$ and $i > 0$. Hence $$\pcotcoh{\lambda}{P} = 0$$ for all $i > 0$ and regular $\lambda \in D$ (recall that $p_P : \pcot{P} \rightarrow G/P$ is the bundle map).
\end{theorem}

\begin{proof} For the moment, assume that $G$ is a semisimple algebraic group over $\algclosed{p}$ with all components of type $A$. Consider the inclusion $$G \times^B \nil{p} \hookrightarrow G \times^B \mf{g} \, \cong \, G/B \times \mf{g} \, .$$ Let $q : G/B \times \mf{g} \rightarrow G/B$ be projection onto the first coordinate. Since $\mf{g}$ is affine, the bundle $q^* \eb{\lambda}{B}$ is ample on $G/B \times \mf{g}$, since $\eb{\lambda}{B}$ is ample on $G/B$. Note that the restriction of $q^* \eb{\lambda}{B}$ to $G \times^B \nil{p}$ is $p_P^* \, \eb{\lambda}{B}$; hence $p_\mf{p}^* \, \eb{\lambda}{B}$ is ample on $G \times^B \nil{p}$.

Note that there is a proper morphism $ G \times^B \nil{p} \rightarrow \mf{g} $ given by $$G \times^B \nil{p} \, \hookrightarrow \, G/B \times \mf{g} \, \twoheadrightarrow \, \mf{g} \, .$$ By Corollary \ref{cor:frobvanishing} and Theorem \ref{th:BK}, we have that $$\Cohom{i}{  G \times^B \nil{p}  }{  p_P^* \, \eb{\lambda}{B}  } = 0$$ for all $i > 0$. The result in characteristic 0 now follows from base change, cf \cite{BK}, section 1.6.
\end{proof}

\section{Examples} \label{sec:examples}

In this section we explicitly compute some examples to illustrate Theorem \ref{th:brygen}. All three examples will be generalized BK-filtrations on $L$-highest weight subspaces of $\irrepgw{0}{\mu}$ for some $\mu \in D$; we will always consider the 0 weight space since it doesn't make the examples any more interesting to consider other weight spaces.

\begin{definition} For $\alpha \in \pi$ let $\chi_\alpha$ (resp. $\chi_\alpha^\vee$) denote the fundamental weight (resp. fundamental coweight) corresponding to $\alpha$. For an enumeration $\{  \alpha_1, \ldots, \alpha_n  \}$ of $\pi$ let $\chi_i$ (resp. $\chi_i^\vee$) denote the fundamental weight (resp. fundamental coweight) corresponding to $\alpha_i$.
\end{definition}

\subsection{Example 1}

In type $A$, there is a bijective correspondence between partitions $\part = [p_1, \ldots, p_k]$ of $n$ where the $p_i$ are nonzero and nonincreasing, and nilpotent orbits in $\nc$. Orbit representatives for these partitions are constructed as follows.

If the partition is such that $p_i = 1$ for all $1 \leq i \leq k$ then the associated orbit is the 0 orbit, so we assume that $\part \neq [1, \ldots, 1]$. For any $\beta \in \Delta^+$ let $X_\beta \in \mf{g}_\beta$ be a Chevalley basis element. Let $\{   \alpha_1, \ldots, \alpha_{n-1}   \}$ denote the simple roots from left to right in the Dynkin diagram of type $A_{n-1}$.

Given a partition $\part$ as above, set $m =\textrm{max} \{  j : p_j > 1 \}$. For each $1 \leq i \leq m$ set $N_i := \displaystyle \sum_{l = 1}^i p_{l-1}$ (where we set $p_{-1} = 0$) and set $$X_i := \displaystyle \sum_{j = N_i + 1}^{N_{i + 1} - 1} X_{\alpha_j} \, .$$ We now obtain an orbit representative corresponding to $\part$ by setting $$X_{\part} := \sum_{i = 1}^m X_ i \, .$$

We may explicitly describe the matrix corresponding to $X_\part$ as follows. For any $q > 0$ let $J(q)$ be the $q \times q$ matrix with $1's$ on the superdiagonal and $0's$ everywhere else. Then $X_\part$ is the matrix with Jordan blocks of size $p_1, \ldots, p_k$, i.e. the block diagonal matrix $$ \textrm{diag} \big(J(1), \ldots, J(q) \big) :=
\begin{pmatrix}
J(p_1)    &  & &  \\
      	&  J(p_2) & & \\
	& & \ddots & \\
	& & & J(p_k) \\
\end{pmatrix} \, .$$

The bijection between partitions and orbits is given by $\part \mapsto \orb{X_{\part}}$. Even nilpotent orbits correspond to partitions consisting of only even or only odd parts.

Let $G$ be of type $A_3$. There are 5 partitions of $n = 4$ and hence 5 nilpotent orbits in $G$. All but one are even; the non-even orbit corresponds to the partition $[2, 1, 1]
$. Let $\pi = \{  \alpha_1, \alpha_2, \alpha_3  \}$ be the standard ordering of the simple roots.

Consider the nilpotent orbit corresponding to the partition $[3, 1]$. By the construction above we see that $X := X_{\alpha_1} + X_{\alpha_2}$ is an orbit representative. However, one may check (cf \cite{CM}) that a semisimple element $H$ in an $sl_2$ triple containing $X$ is given by $H = 2 \chi_1^\vee + 2 \chi_2^\vee - 2 \chi_3^\vee$. As $H$ is not dominant we now conjugate the triple to obtain a representative of our nilpotent orbit in good position.

Choose a representative $\dot{r}_{\alpha_3} \in N(T)$ of $\r{3} \in W$. We may choose $\dot{r}_{\alpha_3}$ so that $\dot{r}_{\alpha_3} (X) = X_{\alpha_1} + X_{\alpha_2 + \alpha_3}$. Then $H' := \dot{r}_{\alpha_3}(H) = 2 \chi_1^\vee + 2 \chi_3^\vee$. Set $Z := \dot{r}_{\alpha_3} (X)$; then $Z$ is in good position.

Now, considering $H'$, we see that the standard parabolic $P$ corresponding to $Z$ is the parabolic corresponding to $\{ \alpha_2 \} \subseteq \pi$. Explicitly, on the Lie algebra level, we have that the Levi of $\mf{p} = \textrm{Lie}(P)$ is $$\mf{l} = \mf{h} \oplus \mf{g}_{\alpha_2} \oplus \mf{g}_{-\alpha_2}$$ and the nilradical of $\mf{p}$ is the $\mf{b}$-subalgebra of $\mf{n}$ with weights $$\{ \alpha_1, \, \alpha_3, \, \alpha_1 + \alpha_2, \, \alpha_2 + \alpha_3, \, 2 \rho = \alpha_1 + \alpha_2 + \alpha_3 \} \, .$$

Now set $\mu := \chi_2 + 2 \chi_3 = \alpha_1 + 2 \alpha_2 + 2 \alpha_3 \in D$ and consider the weight-0 subspace $\Lhi{0}{\mu}{P} \subseteq \irrepgw{0}{\mu} \subseteq \irrepg{\mu}$. Recall that we have the tangent bundle $p: \pcot{P} \rightarrow G/P$ of $G/P$. By Theorem \ref{th:broer} we have $$\pcotcoh{ \chara{0} }{P} \, \cong \, \cohom{i}{\pcot{P}}{  \struct{\pcot{P}}  } = 0$$ for all $i > 0$. Thus, by Theorem \ref{th:brygen}, the jump polynomial for the generalized BK-filtration on $\Lhi{0}{\mu}{P}$ is given by the $P$-analog Kazhdan-Lusztig polynomial $\qana{\mu}{0}$.



Note that $\pkos(\lambda) = 0$ whenever $\lambda$ is not in the $\mb{Z}^+$-span of $\Delta^+ \setminus \Delta^+_L$ (the weights of $\nil{p}$). A quick computation verifies that the only $W  \: *$-translates of $\mu$ in this span are $\mu$, $\r{1} * \mu = 2 \alpha_2 + 2 \alpha_3$, and $\r{2} * \mu = \alpha_1 + 2 \alpha_3$.

One now checks that $\pkos(\mu) = 2 q^2 + q^3$, $\, \pkos(\r{1} * \mu) = q^2$, and $\pkos(\r{2} * \mu) = q^3$. Thus $$\qana{\mu}{0} = \pkos(\mu) - \pkos(\r{1}*\mu) - \pkos(\r{2}*\mu) = q^2 \, .$$ This says that $\Lhi{0}{\mu}{P}$ is 1-dimensional and that for any nonzero $v \in \Lhi{0}{\mu}{P}$ we have $Z^3 . v = 0$ and $Z^2 . v \neq 0$.

\subsection{Example 2}

Let $G$ be of type $G_2$. Let $\alpha$ denote the short simple root and $\beta$ the long simple root. 

For any semisimple Lie algebra $\mf{g}$ there is a unique dense orbit in $\nc \setminus \orb{reg} \, $, where $\orb{reg}$ is the regular nilpotent orbit (cf \cite{CM}). This orbit is called the \textbf{subregular} nilpotent orbit. By Example 8.2.13 in \cite{CM}, the subregular orbit in $G$ is even and the corresponding weighted Dynkin diagram has 2 on the node corresponding to $\alpha$ and 0 on the node corresponding to $\beta$. Set $H := 2 \chi_\alpha^\vee \in \mf{h}$; then there is an $sl_2$-triple $\{   X, Y, H  \}$ with $X \in \orb{subreg} \, \cap \, \nil{p}$, and $X$ is clearly in good position. 

The corresponding parabolic $P$ is the minimal parabolic corresponding to the long simple root $\beta$ with Levi factor and unipotent radical (on the Lie algebra level) given by $$\mf{l} = \mf{h} \oplus \mf{g}_{\beta} \oplus \mf{g}_{-\beta}$$ and $$  \nil{p} = \bigoplus_{ \beta \in \Delta^+ \setminus \Delta^+_L } \mf{g}_\beta  \, ,$$ respectively, where $$\Delta^+ \setminus \Delta^+_L = \{   \alpha, \, \alpha + \beta, \, 2 \alpha + \beta, 3 \alpha + \beta, \, 3 \alpha + 2 \beta   \} \, .$$ 

Consider the adjoint representation $V := \irrepg{\mu}$ of $G$ where $$\mu = 3\alpha + 2\beta = \chi_\beta \, . $$ We compute the generalized BK-filtration on the subspace $\Lhi{0}{\mu}{P} \subseteq \irrepgw{0}{\mu} \subseteq \irrepg{\mu}$. This means we should compute $\qana{\mu}{0}$. Hence we must compute $\pkos(w*\mu - 0) = \pkos(w*\mu)$ for each $w \in W$. One verifies easily that for $w \in W$, $w*\mu \in \mb{Z}^+ (\Delta^+ \setminus \Delta^+_L)$ iff $w \in \{   e, r_\alpha, r_\beta   \}$.

One now checks that $\pkos(e * \mu) = q + q^2 + q^3 \, ,$ $\pkos(r_\alpha * \mu) = q^2 \, ,$ and $\pkos(r_\beta * \mu) = q^3$. Hence we obtain $$\qana{\mu}{0} = q + q^2 + q^3 - q^2 - q^3= q \, .$$

\newpage


\bibliography{thebibliography}
\end{document}